\documentclass{amsart}

\usepackage[dvipdfmx]{graphicx}

\newtheorem{thm}{Theorem}[section]
\newtheorem{cor}[thm]{Corollary}

\newtheorem{lem}[thm]{Lemma}

\theoremstyle{remark}
\newtheorem{rem}[thm]{Remark}

\theoremstyle{definition}
\newtheorem{dfn}[thm]{Definition}

\numberwithin{equation}{section}
\numberwithin{thm}{section}

\theoremstyle{plain}
\newtheorem*{clm}{Claim}

\begin{document}


\title{Thin position for incompressible surfaces in 3-manifolds}

\author{Kazuhiro Ichihara}
\address{Department of Mathematics, College of Humanities and Sciences, Nihon University, 3-25-40 Sakurajosui, Setagaya-ku, Tokyo 156-8550, Japan.}
\email{ichihara@math.chs.nihon-u.ac.jp}

\author{Makoto Ozawa}
\address{Department of Natural Sciences, Faculty of Arts and Sciences, Komazawa University, 1-23-1 Komazawa, Setagaya-ku, Tokyo, 154-8525, Japan.}
\email{w3c@komazawa-u.ac.jp}

\author{J. Hyam Rubinstein}
\address{Department of Mathematics and Statistics, The University of Melbourne, VIC 3010, Australia.}
\email{rubin@ms.unimelb.edu.au}

\thanks{The first author and the second author are partially supported by Grant-in-Aids for Scientific Research (C) (No. 26400100 and 26400097), The Ministry of Education, Culture, Sports, Science and Technology, Japan, respectively. The third author is partially supported under the Australian Research Council's Discovery funding scheme (project number DP130103694).}

\begin{abstract}
In this paper, we give an algorithm to build all compact orientable atoroidal Haken 3-manifolds with tori boundary or closed orientable Haken 3-manifolds, so that in both cases, there are embedded closed orientable separating incompressible surfaces which are not tori.
Next, such incompressible surfaces are related to Heegaard splittings. For simplicity, we focus on the case of separating incompressible surfaces, since non-separating ones have been extensively studied. 
After putting the surfaces into Morse position relative to the height function associated to the Heegaard splittings, 
a thin position method is applied so that levels are thin or thick, depending on the side of the surface. 
The complete description of the surface in terms of these thin/thick levels gives a hierarchy. 
Also this thin/thick description can be related to properties of the curve complex for the Heegaard surface. 
\end{abstract}

\maketitle

\section{Introduction}

We give an algorithm to build all compact orientable atoroidal Haken 3-manifolds with tori boundary and closed orientable Haken 3-manifolds, so that in both cases, there are embedded closed orientable incompressible surfaces which are not tori. 
The algorithm can also be viewed as a decomposition result showing how such manifolds can be built from handlebodies and compression bodies with {\it suitable} boundary patterns. 
Given such a description, many properties of the 3-manifolds can be deduced, such as estimates on the Heegaard genus, annularity properties of incompressible surfaces etc. 
The original notion of Haken manifold was introduced by Haken in \cite{Ha}. 

Next, we relate such incompressible surfaces to Heegaard splittings. 
For separating incompressible surfaces, the first step is to put the surfaces into Morse position relative to the height function associated to the Heegaard splittings. 
Then a thin position method is applied so that levels are thin or thick, depending on the side of the surface. 
The thin levels can be viewed as incompressible spanning surfaces, i.e., part of the hierarchy described in the earlier sections. 
So the complete description of the surface in terms of these thin/thick levels gives a hierarchy. 
The special case of Heegaard genus 2 is treated as an illustration of the general theory. 

Finally we relate this thin/thick description to properties of the curve complex for the Heegaard surface. 

\section{Preliminaries}

We give a very brief list of basic definitions and concepts. For more details of 3-manifold topology, see \cite{He} and on thin position see \cite{Sc}. We work throughout in the PL category and all manifolds, surfaces and maps are PL. 

\begin{dfn}

A compact orientable 3-manifold is \textit{irreducible} if every embedded 2-sphere bounds a 3-ball.

\end{dfn}

\begin{dfn}
A closed embedded orientable surface $S$ in a compact orientable 3-manifold is \textit{incompressible} if $S$ is not a 2-sphere and the homomorphism induced by inclusion $\pi_1(S) \to \pi_1(M)$ is one-to-one. 

\end{dfn}

\begin{dfn}
A compact orientable irreducible 3-manifold $M$ is called \textit{atoroidal} if the only embedded incompressible torus is parallel to a component of $\partial M$. 

\end{dfn}

\begin{dfn}

A \textit{handlebody} is a compact orientable 3-manifold with a single boundary component, which has an embedded graph (spine) which is a homotopy retract. 

\end{dfn}

\begin{dfn}

A \textit{compression body} is obtained by attaching 2-handles to the boundary surface $S \times \{0\}$ of a product $S \times [0,1]$, where $S$ is a closed orientable surface of genus at least two.
We refer to $S \times \{1\}$ as the \textit{outer boundary} and the other boundary components as the \textit{inner boundary} of the compression body. None of the inner boundary surfaces are 2-spheres.
We denote the outer boundary by $S$ rather than $S \times \{1\}$.

\end{dfn}

\begin{dfn}
A \textit{Heegaard splitting} of a compact orientable 3-manifold $M$ is a closed orientable surface $S$ embedded in $M$ so that splitting $M$ along $S$ gives two regions which are handlebodies or compression bodies. 

\end{dfn}

\begin{dfn}

The \textit{Hempel distance} of a Heegaard splitting $S$ for $M$ is defined as follows. Consider the collections of curves $\mathcal C$ and $\mathcal C^{\prime}$ which bound compressing disks for $S$ in the two regions, which are handlebodies or compression bodies on either side of $S$. A path between these collections is a sequence of essential simple closed curves $C=C_0,C_1, \dots C_k$ so that each pair $C_i, C_{i+1}$ are disjoint and $C_0  \in \mathcal C, C_k \in \mathcal C^{\prime}$. The Hempel distance is then the smallest value of $k$ amongst all such sequences. 

\end{dfn}

\begin{dfn}

A Heegaard splitting $S$ for $M$ is \textit{strongly irreducible} if every compressing disk on one side of $S$ meets every compressing disk on the other side of $S$. 

\end{dfn}

\section{Boundary patterns on handlebodies and compression bodies}

We start with the concept of suitable boundary pattern which comes from \cite{CR}. 
We extend this to the case of compression bodies. 

\begin{dfn} 
Suppose that $H$ is a handlebody or compression body. 
Assume that $S$ is a surface equal to the outer boundary of $H$, if $H$ is a compression body, or the whole boundary $S=\partial H$ if $H$ is a handlebody. 
Let $S$ be divided into two subsurfaces $P, \tilde P$ so that $\partial P = \partial \tilde P$, $S=P \cup \tilde P$, all the curves of $\partial P$ are essential in $S$ and any compressing disk for $H$ must intersect each of $P, \tilde P$ in at least two essential arcs. 
Then we say that the pair $\{P, \tilde P\}$ is a \textit{suitable boundary pattern} for $H$. 
\end{dfn}

The following lemma comes from \cite{CR}. 
We give a proof here with the extension to compression bodies, as this is a crucial result for our algorithm. 

\begin{lem} Suppose that $H$ is a handlebody or compression body. 
Two subsurfaces $\{P, \tilde P\}$ form a suitable boundary pattern for $H$, where $S = \partial H$ or the outer boundary of $H$, depending on whether $H$ is a handlebody or compression body, if and only if the following conditions hold. 
Firstly $S=P\cup \tilde P$,  $\partial P = \partial \tilde P$ and every curve of $\partial P$ is essential in $S$. 
Secondly there is a complete collection $\mathcal D$ of compressing disks for $H$ with the property that every disk $D$ in $\mathcal D$ intersects each of $P, \tilde P$ in at least two essential arcs and selecting any compressing disk $D^*$ disjoint from $\mathcal D$, a replacement of a disk in $\mathcal D$ by $D^*$ cannot give a new complete collection of compressing disks which have total numbers of arcs of intersection with $P, \tilde P$ less than that for $\mathcal D$.
\end{lem}

\begin{proof}
Let us suppose that $\{P, \tilde P\}$ is a suitable boundary pattern for $H$, and show that it satisfies the two conditions described in the statement. 
Then, the first condition is included in the definition of a suitable boundary pattern. 
To show the second condition, we take any complete collection of compressing disks for $H$ so that the total numbers of arcs of intersection with $\{P, \tilde P\}$ is minimal.
This collection satisfies the second condition. 

To show the other direction, suppose that subsurfaces $\{P, \tilde P\}$ of $S$ satisfy the conditions in the lemma. 
Let us prove this is a suitable boundary pattern. 
So suppose to the contrary, that some compressing disk $D^\prime$ meets each of $P, \tilde P$ in one essential arc or is disjoint from one of these subsurfaces. 
Consider the intersection of $D^\prime$ with the family $\mathcal D$. 
As usual, we can eliminate all loops of intersection by cutting and pasting. 
Assume next that  $\lambda$ is an outermost arc of intersection between $D^\prime$ and some disk $D_i$ in $\mathcal D$ so that $\lambda$ cuts off a bigon $D_0$ on $D^\prime$ with interior disjoint from $\mathcal D$. 
Let $D_1,D_2$ be the bigons obtained by splitting $D_i$ along $\lambda$. 
If either of the disks $D_0 \cup D_1, D_0 \cup D_2$ is inessential, then an isotopy of $D_i$ reduces the number of arcs of intersection with $D^\prime$. 
We can assume therefore that neither of the disks $D_0 \cup D_1, D_0 \cup D_2$ is inessential, by isotoping the family $\mathcal D$ until the number of arcs of intersection with $D^\prime$ is minimal. 

Notice that there must be at least two outermost bigons on $D^\prime$ and hence one of these bigons must intersect $\partial P$ in at most one point, since there are at most two such points on $\partial D^\prime$ by assumption. 
Hence we see that if both $D_1,D_2$ intersect $\partial P$  in at least two points, then a replacement of $D_i$ by either $D_0 \cup D_1, D_0 \cup D_2$ reduces the number of arcs of intersection with the boundary pattern. 
By the definition of a complete family of disks, at least one of these two replacements must be a new complete family of disks for $H$, so this contradicts our hypotheses in the lemma. 
The conclusion is that at least one of $D_1,D_2$ crosses $\partial P$  in at most one point. 

But now, one of the disks $D_0 \cup D_1, D_0 \cup D_2$ is a compressing disk for $H$ disjoint from the family $\mathcal D$ and which intersects $\partial P$  in at most two points. 
We can replace one of the disks of $\mathcal D$ by this disk and reduce the number of intersections with the boundary pattern, which is again a contradiction. 
This completes the proof of the lemma. 
\end{proof}

\begin{rem}
Note the algorithmic nature of the lemma, giving a bounded process to check if a given boundary pattern is suitable or not. 
For we can start with any boundary pattern on a handlebody or compression body. 
Pick a complete family of compression disks and isotope them to remove any inessential arcs of intersection with the boundary pattern. 
If any disks in the family intersect the boundary pattern in at most two arcs, then the boundary pattern is not suitable. 
Otherwise, one has to check whether any replacement can be done which decreases the number of intersections with the boundary pattern. 
There is an algorithm to check if such a replacement exists, since we need only search for compressing disks disjoint from the family with fewer intersections than those in the family.  
\end{rem}

A quick summary is as follows. 
If we split $H$ open along a complete family of compressing disks, then the boundary pattern becomes a system of arcs joining pairs of disjoint simple closed curves on a 2-sphere, where each compressing disk becomes a pair of circles. 
Normal curve theory can be used to list all simple closed curves on this 2-sphere, which are disjoint from all the simple closed curves and meet the system of arcs in fewer points than at least one of the disks. 
If such a curve separates the two circles representing a compressing disk and has fewer intersections with the arc system than this disk, then a replacement is possible. 
Otherwise no replacement can be achieved. 

\section{The decomposition algorithm}

The key idea is to use very short hierarchies as in \cite{AR} to decompose a $3$-manifold in our class into compression bodies and handlebodies with suitable boundary patterns. 
We can then view the process of building all our $3$-manifolds as starting with suitable boundary patterns on a collection of handlebodies and compression bodies and then gluing subsurfaces in pairs. 
An interesting observation will be that {\it any} gluing of pairs of subsurfaces is allowable to produce a $3$-manifold in our class. 
In particular, a given collection of handlebodies and compression bodies with suitable boundary pattern, where the subsurfaces can be matched in pairs, produces an infinite number of $3$-manifolds in our class. 

\begin{thm}
Let $M$ be a compact orientable irreducible atoroidal $3$-manifold which is either closed or has incompressible tori boundary. In both cases, assume $M$ has a closed embedded separating incompressible surface which is not boundary parallel. Choose a maximal collection $\mathcal S$ of disjoint embedded separating incompressible surfaces for $M$ which are not boundary parallel and not parallel to each other. 
Then there is a decomposition of $M$ into a collection of handlebodies and compression bodies with suitable boundary pattern. 
In fact, there is a collection of spanning surfaces $\mathcal S^*$ with the following properties. 
Firstly each region $R$ obtained by cutting $M$ open along $\mathcal S$ has spanning surfaces which are incompressible and boundary incompressible surfaces with boundary on each of the surfaces in $R \cap \mathcal S$. 
These spanning surfaces have the property that they do not separate $R$ so that when $R$ is cut open along the spanning surfaces, the result is either a handlebody or a compression body, where the inner boundary surfaces are tori in $\partial M$. 
The boundary pattern arises by taking subsurfaces in $P$ which are copies of the spanning surfaces and subsurfaces in $\tilde P$ which are in $\mathcal S$. 
\end{thm}

\begin{proof}
It suffices to construct the spanning surfaces inside a single region $R$ obtained by cutting $M$ open along $\mathcal S$. 
Assume first that $R$ has all boundary surfaces which are copies of surfaces in $\mathcal S$. 
In this case, we follow the argument in \cite{AR}. 
It suffices to find a homomorphism $\phi$ from $H_1(R,\mathbb{Q})$ onto $\mathbb{Q}$ so that every component $S_i$ of $\partial R$ has the property that the image of the inclusion map $H_1(S_i,\mathbb{Q}) \rightarrow H_1(R,\mathbb{Q})$ has image which is not in the kernel of $\phi$. 
Once $\phi$ has been constructed, then we follow the argument of Stallings \cite{S}. 
Namely we can construct a map $f:R \rightarrow S^1$ so that the induced homomorphism on first homology is $\phi$. 
Then surgering the pullback of a point $x_0$ in $S^1$, we can replace $f$ by a homotopic map, again denoted by $f$, with the property that $f^{-1}(x_0)$ is collection of required spanning surfaces. 

To build $\phi$, note that as is well known, each component $S_i$ of $\partial R$ has induced inclusion $H_1(S_i,\mathbb{Q}) \rightarrow H_1(R,\mathbb{Q})$ with image $J_i$ having rank at least the genus of $S_i$. 
So one can build a homomorphism $\phi_i:H_1(R,\mathbb{Q}) \rightarrow \mathbb{Q}$ so that the image of $J_i$ is non zero. 
Then, by taking an appropriate linear combination of these maps $\phi_i$, we get the homomorphism $\phi$ required. 

The final step is straightforward, once we have built the spanning surfaces, cutting $R$ open along them must give a handlebody. 
For we get a connected $3$-manifold and compressing the boundary cannot give any closed separating incompressible surfaces which are not boundary parallel in $R$, for otherwise we would contradict the maximality of the collection of surfaces in $\mathcal S$. 
The conclusion is that the boundary is completely compressible and so the manifold must be a handlebody. 
Finally notice that the boundary pattern induced as in the statement of the theorem, is indeed a suitable boundary pattern. 
For if we had any compressing disks meeting the boundary pattern in fewer then four arcs, then either a spanning surface would be boundary compressible, or a spanning surface or surface in $\mathcal S$ would be compressible. Since neither is the case, this completes the discussion of the first case, once $\phi$ is built. 

Next, in the case that $R$ has some boundary components in $\partial M$, the only difference from the previous case is that we need a homomorphism $\psi:H_1(R,\mathbb{Q}) \rightarrow \mathbb{Q}$ with the properties that for each component $S_i$ of $\partial R$ which is a copy of a surface in $\mathcal S$, the image $J_i$ of the mapping $H_1(S_i,\mathbb{Q}) \rightarrow H_1(R,\mathbb{Q})$ is not in the kernel of $\psi$, whereas if $S_i$ is a boundary torus of $\partial M$, then $J_i$ is in kernel $\psi$. 

The construction of $\psi$ follows a similar pattern to the previous argument. 
The main difference is that we can find a mapping $\psi_i$ for boundary surfaces in both $\partial R$ and in $\mathcal S$, so that $\psi$ maps each image of the first homology of a boundary torus to zero. 
This follows by noting that the image $J_i$ of the mapping $H_1(S_i,\mathbb{Q}) \rightarrow H_1(R,\mathbb{Q})$ always contains elements not in the images of the first homology of the boundary tori. 
So it is straightforward to find such mappings $\psi_i$ and then take a linear combination to find $\psi$.

As in the previous case, when we cut $R$ open along all the spanning surfaces, we obtain a manifold with two types of boundary surfaces. 
One type comes from surfaces in $\mathcal S$ cut open along spanning surfaces. 
The other type are boundary tori. 
Compressing the first types of surface must result in a collection of products $T^2 \times [0,1]$, one for each boundary torus, since any other possibility will contradict the maximality of the family $\mathcal S$. 
Proving the boundary pattern is suitable is the same as above. 
Note that the first type of boundary surfaces form the outer boundary and the boundary tori form the inner boundary of the compression body. 
This completes the proof. 
\end{proof}

\section{Heegaard diagrams}

We are interested first in the case where $S$ is a closed orientable surface of genus $2$ and two 2-handles are attached to $S \times [0,1]$ along separating essential curves $C \subset S \times \{0\}$ and $C^\prime \subset S \times \{1\}$. 
If we project both curves $C,C^\prime$ onto $S$, abusing notation by using the same symbols for the projected curves, then $|C \cap C^\prime| =4k$, where $k$ is a positive integer. Here we also assume that the projected curves cross transversely and minimally. 
Let $\mathcal C$ denote the curve complex for $S$. 
We are interested in paths $C_0,C_1, \dots C_{3k}$ in $\mathcal C$, where the three curves $C_{3i}, C_{3i+1}, C_{3i+2}$ are disjoint and non parallel and hence form a 2-simplex in the curve complex, with $C=C_0, C^\prime =C_{3k}$. 
We also want the three curves $C_{3i}, C_{3i-1}, C_{3i-2}$ to be disjoint and non parallel. 
Moreover we require that each $C_{3i}$ is separating, both $C_{3i+1}$ and $C_{3i+2}$ are non separating and $|C_{3i} \cap C^\prime|$ and $2|C_{3i+1} \cap C^\prime|$, $2|C_{3i+2} \cap C^\prime|$ are all decreasing, as $j=3i,3i+1,3i+2$ increases. 
Note that the simplest such a path of curves arises  when each pair $C_{3i}, C_{3(i+1)}$ meets in exactly four points, for $0 \le i \le k-1$. We will focus on this case. 
Finally we order the curves so that $C_{3i-2}$ and $C_{3i+1}$ are on the same side of $C_{3i}$ for each $1 \le i \le k-1$.
See Figure \ref{fig} for simple example.

\begin{figure}[htbp]
	\begin{center}
	\includegraphics[trim=0mm 0mm 0mm 0mm, width=.5\linewidth]{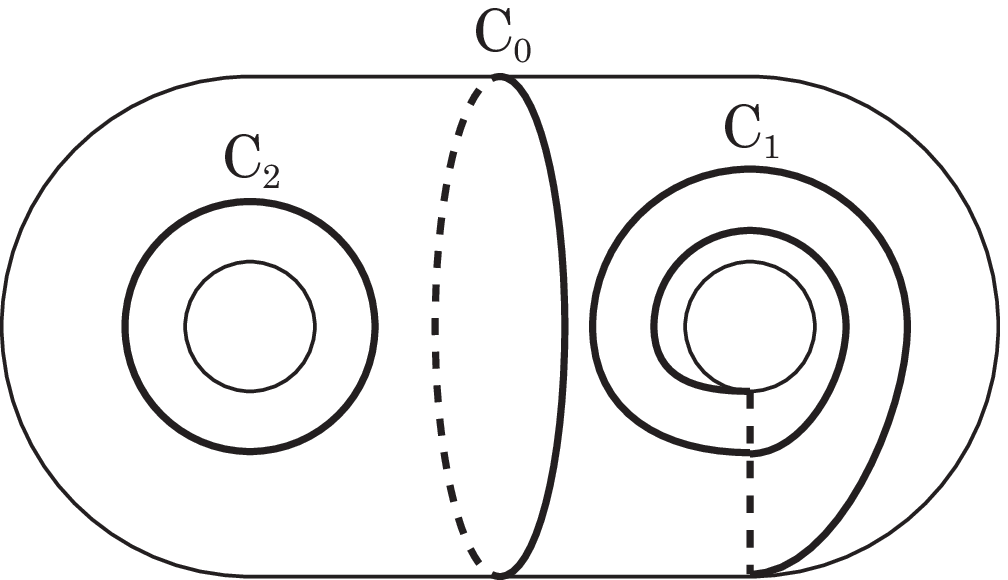}
	\includegraphics[trim=0mm 0mm 0mm 0mm, width=.5\linewidth]{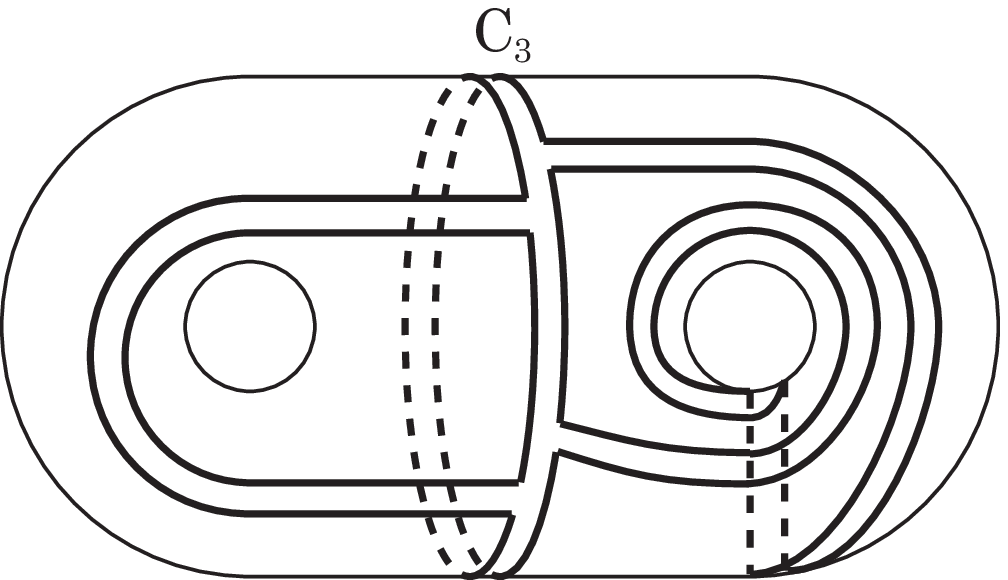}
	\end{center}
	\caption{}
	\label{fig}
\end{figure}

Note that this simple Heegaard diagram $C, C^\prime$ has the following form. 
We can view $C$ as a standard essential separating curve splitting $S$ into two once-punctured tori $T_+,T_-$. 
Moreover the intersections $C^\prime \cap T_+, C^\prime \cap T_-$ are then families of essential arcs. 
The slopes of these arcs are not important, just the number of each type. 
So we can label these arcs by three non-negative integers $n_1,n_2,n_3$ for $C^\prime \cap T_+$ and $m_1,m_2,m_3$ for $C^\prime \cap T_-$.

Next, the collection of arcs of $C^\prime \cap T_+,$ attached to $C$ can be viewed as a circle with a collection of arcs joining pairs of points and similarly for $C^\prime \cap T_-$. 
A rotation $R$ of $C$ is required to determine how to match up these two families. We do this by rotating an end point of some chosen arc of $C^\prime \cap T_+$ so that it matches a similar end point of a chosen arc of $C^\prime \cap T_-$. 
Our convention is that the base points are chosen on $C$ corresponding to the first arcs in the families labelled $n_1,m_1$ and the rotation $R$ is done clockwise. 
We can then refer to the Heegaard diagram ${\mathcal H}=(n_1,n_2,n_3,m_1,m_2,m_3,R)$.
Note that there are obvious necessary conditions for such a diagram to give a single separating curve, namely $n_1 + n_2 + n_3 = m_1 + m_2 + m_3$ and the sum of $n_i$'s is even.

\begin{rem}
For the above procedure, the following questions arise naturally. 

 \begin{itemize}
 
 \item Find sufficient conditions on $\mathcal H$ to ensure that the resulting curve is separating and connected. 
 It would be nice to even have an iterative scheme to build such diagrams. A possible approach is  by band summing. Namely taking two parallel arcs in $T_+$ and removing them, so long as there are parallel arcs in $T_-$ which have ends on either side of the first pair, then we can delete both pairs and if the first diagram gives a connected curve, so does the second.
 
 \item Find a suitable generalization to the case of arbitrary genus $g$. 
 Possibly this might work better with a family of $g-1$ separating 2-handles attached to each of $S \times \{0\}$ and to $S \times \{1\}$. So this would give a family of $g-1$ separating surfaces - see below. 
 
 \end{itemize}
 
 \end{rem}

\section{Paths and separating incompressible surfaces}

As in the previous section, 
let $M$ be a manifold obtained by adding a pair of 2-handles to $S \times [0,1]$ along its two boundary components to curves $C, C^\prime$.
We put some additional conditions on the path $C_0,C_1, \dots C_{3k}$ in $\mathcal C$, which are sufficient to build a pair of separating incompressible surfaces in Morse position relative to the height function on $M$. 

\begin{thm}
Let $M$ be the manifold obtained by attaching two 2-handles to $S \times [0,1]$ along separating essential curves $C \subset S \times \{0\}$ and $C^\prime \subset S \times \{1\}$, where $S$ is a closed orientable surface of genus $2$. 
Consider a path $C_0,C_1, \dots C_{3k}$ in the curve complex $\mathcal C$ for $S$ as in the previous section. 
Suppose that $C_{3i}$ intersects $C_{3(i+1)}$ in four points, the two curves $C_{3i-2}, C_{3i+1}$ on the same side of $C_{3i}$ must intersect at least twice, and the two curves $C_{3i-1}, C_{3i+2}$ on the other side of $C_{3i}$ must also meet at least twice, for $0 \le i \le k-1$. 
Then $M$ contains two disjoint closed orientable separating incompressible surfaces $J, J^\prime$. 
\end{thm}

\begin{proof}
Two disjoint surfaces $J, J^\prime$ in $M$ can be build from $\mathcal C$ as follows. 
Start with two parallel copies of the disk representing the 2-handle attached at $C=C_0$ at a level $S \times \{t_1\}$, where $t_1 =0$. 
Next we glue on a pair of pants with one curve at $C_0$ and the other two curves along parallel copies of $C_1$ (respectively $C_2$) at a level $S \times \{t_2\}$. 
Next  two pairs of pants are glued on with the first having two boundary curves to the copies of $C_1$ and the second using two copies of $C_2$ at a level $S \times \{t_2\}$ and two copies of $C_3$ at a level $S \times \{t_3\}$. 
We continue on until eventually two pairs of pants are glued to two copies of $C_{3k-2}$ and two copies of $C_{3k-1}$ at a level $S \times \{t_{k-1}\}$ and to two copies of the curve $C^\prime = C_k$ at at a level $S \times \{t_k\}$, where $t_k=1$. 
Note that $0=t_1<t_2< \dots t_k=1$. 
So this completes the construction of $J, J^\prime$. 

To show that $J, J^\prime$ are incompressible, by using standard inner-most arguments, 
it suffices to prove that there cannot be a compressing disk in any of the three regions in the complement of $J \cup J^\prime$. 
To summarize the argument, assume there is such a disk, say $D$. 

Note that each of the three regions has a collection of annuli with boundary on $J, J^\prime$. 
The annuli have boundary curves given by the parallel copies of $C_0, C_1, \dots, C_{3k}$, which are essential on $J$ except for $C_0$ and $C_{3k}$, and so are incompressible except for the top and bottom ones. 
In fact, the region between $J$ and $J^\prime$ has annuli with boundaries on copies of $C_0,C_3, \dots C_{3k}$, the region bounded by $J$ has annuli with boundaries on copies of $C_1,C_4, \dots, C_{3k-2}$ and the final region bounded by $J^\prime$ has annuli with boundaries on \newline $C_2,C_5, \dots, C_{3k-1}$.

\begin{clm}
The annuli are not boundary compressible in the regions. 
\end{clm}

\begin{proof}
Suppose that there was such a boundary compression disk, which can be assumed to be a disk with one boundary arc on such an annulus, one on $J$ or $J^\prime$ and interior disjoint from all the annuli, $J$ and $J^\prime$. 
Now the disk must lie in a handlebody of genus 2 or a region which is of the form torus $\times [0,1]$. The reason is that these are the regions formed by cutting $M$ open along $J$ or $J^\prime$ and all the annuli.
The latter type of region is a product, so we see that there are no such a boundary compression disk in such a region. 
The former is a region bounded by $J$ or $J'$ having the annuli with boundaries on copies of $C_1,C_4, \dots, C_{3k-2}$ or $C_2,C_5, \dots, C_{3k-1}$, respectively. 
This region can be regarded as a punctured torus $\times [0,1]$, with copies of the annuli on the top and the bottom surfaces. 
Our strategy here is to apply Lemma 1. 
Namely if we have a complete collection of compressing disks for a genus 2 handlebody region, with the property that any disjoint compressing disk meets the boundary pattern at least as many times as the disks in the family, then the boundary pattern is disk busting. 
In our situation, the way to implement this is to think of two copies of each of the curves $C_{3i-2}, C_{3i+1}$ or $C_{3i-1}, C_{3i+2}$ as the boundary pattern and the disks of the form arc $\times [0,1]$ as the complete family. 
Then the assumption that the curves $C_{3i-2}, C_{3i+1}$ intersect twice, implies that there cannot be a disk which crosses an annulus with two boundary curves parallel to $C_{3i-2}$ but does not meet $C_{3i+1}$ and similarly for the case of $C_{3i-1}, C_{3i+2}$. 
So this proves our assertion about the annuli. 
\end{proof}

Therefore any such a disk $D$ can be isotoped off these annuli and so must lie between levels of the form $S \times \{t_i\}, S \times \{t_{i+2}\}$, for $0 \le i \le 3k-2$ or in the region above the level $S \times \{t_1\}$ or below the level $S \times \{t_{3k-1}\}$. 
The latter are trivial regions which are 2-handles or products of the form torus $\times [0,1]$.

To finish, we need to prove that there are no compressing disks inside these regions. 
Now the former type are easily seen to be handlebodies of genus 2 or genus 3 with a natural product structure of the form once-punctured torus $\times [0,1]$ or four punctured sphere  $\times [0,1]$, depending on whether the region has boundary one surface or two surfaces. 
For the first type of region, the condition that $C_{3i-2}$ and $C_{3i+1}$ meet at least twice, implies that the any compressing disk, which is of the form arc $\times [0,1]$, where the arc has both ends on $C_{3i}$, must meet these curves and so cannot lie entirely on $J$. 
It suffices to show then that there cannot be a general compressing disk which misses both the curves $C_{3i-2}, C_{3i+1}$ or $C_{3i-1}, C_{3i+2}$. 
To do this, we use Lemma 1 again. 
In our situation, again, the way to implement this is to think of two copies of the curves $C_{3i-2}, C_{3i+1}$ or $C_{3i-1}, C_{3i+2}$ as the boundary pattern and the disks of the form arc $\times [0,1]$ as the complete family. 

For the second type of region, we repeat the argument using the boundary pattern consisting of two copies of each of $C_{3i},C_{3(i+1)}$. 
Again the disks are of the form arc $\times [0,1]$ as the complete family, where the arcs have ends on the four punctured sphere, i.e the curves $C_{3i+1}, C_{3i+2}$. 
Each disk crosses the boundary pattern at least four times and the lemma above applies as previously. 
In fact, there is a complete system of quadrilateral disks of this form which meet each curve exactly once. 
So in fact, this region has a product structure by gluing these quad disks together and so $J, J^\prime$ are actually parallel. 
So this completes the proof that both $J$ and $J^\prime$ are incompressible. 
\end{proof}

\begin{rem}
The same method could be used to consider families of surfaces in higher genus surface $\times [0,1]$ with 2 handles attached and also individual surfaces with more complex conditions implying the disk busting conditions. An interesting challenge is to see if such an approach is strong enough to show that a `generic' Heegaard splitting gives a 3-manifold containing a separating incompressible surface. 

\end{rem}

\begin{rem}

Next,  consider the result of Dehn filling of each of the 4 boundary tori of the 3-manifold $M$ in the above theorem. Clearly, the separating incompressible surface constructed has `accidental parabolics' on each of the four cusps. We can apply a well-known result of Y. Q. Wu \cite{wu}. There it is shown that any closed 3-manifold $M^\prime$ obtained by Dehn filling $M$ along a curve in each cusp, which meets the (unique) accidental parabolic for the surface at least twice, then the separating surface remains incompressible in $M^\prime$. So this gives a large number of examples of closed 3-manifolds $M^\prime$ of Heegaard genus two possessing separating incompressible surfaces. Moreover it is easy to see that most of the these examples are rational homology 3-spheres and hence have no non-separating incompressible surfaces. 

\end{rem}

\section{Thin position}

Suppose we start with a separating closed orientable incompressible surface $J$ and a strongly irreducible Heegaard splitting $S$ of a closed irreducible orientable 3-manifold $M$ or compact irreducible orientable 3-manifold $M$ with incompressible tori boundary. 
We would like to reverse the process of the previous sections and write levels of $S$ as spanning surfaces to complete a hierarchy of $M$ starting with $J$. 
Note that we do not require $J$ to be connected, so it could be a separating family of surfaces, with individual members which are non-separating. We use the notation $S_t$, $0<t<1$, as the singular foliation of $M$ by copies of $S$. . We will denote the two handlebodies or compression bodies obtained by splitting $M$ open along $S_t$ by $H^t_1, H^t_2$. As $t \to 0,1$, $H^t_1$ or $H^t_2$ respectively will converge to a graph or a graph connected to some of the tori boundary components of $M$. 

For $t$ small enough,  we can assume that $J$ meets $H^t_1$ in a family of meridian disks. 
Next denote the two sides of $J$ by $M_+,M_-$. 
We can initially apply all possible boundary compressions of $J \cap H^t_2$ so that a band of $S_t$ gets pushed across $J$ from $M_-$ to $M_+$. (See the Appendix to \cite{Ha} for a very elegant discussion of this procedure). 
The effect is to make $S_t \cap M_-$ thin and $S_t \cap M_+$ thick. 
Fix this copy of $S_t$ as level one and denote it by $S_{t_1}$ with the initial position of $S$ as $S_{t_0}$, where $J$ meets $H^{t_0}_1$ in meridian disks. 
 
Now repeat the process for $H^{t_1}_2$ bounded by $S_{t_1}$, but this time interchanging the roles of $M_+,M_-$ so that bands of $S$ gets pushed across $J$ from $M_+$ to $M_-$. 
This will give a new level $S_{t_2}$ for which $S_{t_2} \cap M_-$ is thick and $S_{t_2} \cap M_+$ thin. 
We iterate until eventually $J$ meets a handlebody or compression body corresponding to $H^t_2$ in meridian disks only, for $t$ close to $1$.
Call this level $S_{t_k}$ and assume that $t_0 = \epsilon, t_k=1-\epsilon, t_1 < t_2 < \dots <t_k$, for $\epsilon$ sufficiently small. 

Note that our surfaces $J$ and $J^\prime$ above are in thin position in exactly this sense. 
We call the intersections of some level surface $S_t$ with the sides of $J$ $M_+,M_-$ the sides of $S_t$ relative to $J$. As usual when we put $J$ into Morse thin position relative to the singular foliation corresponding to $S$, this means there are a finite number of critical levels $\hat t$, for $0 < \hat t <1$, so that at such a level there is a single saddle critical point.

Our first observation is that there must be at least one thin surface which is incompressible. 

\begin{thm}
Suppose that $J$ is separating and incompressible and $S$ is a strongly irreducible Heegaard splitting. Denote the two sides of $J$ as $M_+, M_-$. Then either;
\begin{itemize}

\item  there is some non critical level $S_t$ so that $S_t \cap M_+$  is incompressible and $S_{t} \cap M_-$ has compressing disks on both sides of $S_{t}$, or the same with $M_+, M_-$ interchanged.

\item there is a critical level $\hat t$ so that $S_{t} \cap M_+$ is incompressible for $t<{\hat t}$ and $t$ close to ${\hat t}$, and $S_{t} \cap M_-$  is incompressible for $t>{\hat t}$ and $t$ close to ${\hat t}$, or the same with $M_+$, $M_-$ interchanged.

\item there is a critical level ${\hat t}$ so that both $S_{t} \cap M_+$ and $S_{t} \cap M_-$ are incompressible for $t>{\hat t}$ and $t$ arbitrarily close to ${\hat t}$.
\end{itemize}

\end{thm}

\begin{proof}
Suppose at some level $S_t$, one of $S_t \cap M_+$ or $S_t \cap M_-$ has compressing disks on both sides. 
Since $S$ is strongly irreducible, it is immediate that either $S_t \cap M_-$ or $S_t \cap M_+$, respectively is incompressible. 
So the first case of the theorem holds. 

On the other hand, assume that neither $S_t \cap M_+$ nor $S_t \cap M_-$ has compressing disks on both sides, for any value of $t$. 
We know that for $t$ small, there are compressing disks for $S_t \cap M_+$ or $S_t \cap M_-$ in $H_1$, whereas for $t$ close to $1$, there are compressing disks for $S_t \cap M_+$ or $S_t \cap M_-$ in $H_2$. 
At some level $S_u$ this must switch over in the sense that for $t<u$, any compressing disk for $S_t \cap M_+$ or $S_t \cap M_-$ must be in $H_1$, whereas for $t>u$, any compressing disk for $S_t \cap M_+$ or $S_t \cap M_-$ must be in $H_2$. 
The level $u$ must be a critical level ${\hat t}$ at which a band sum occurs which produces the first compressing disk for $S_t \cap M_+$ or $S_t \cap M_-$ in $H_2$, for $t > \hat t$. 
Suppose that there was a compressing disk for $S_t \cap M_+$ or $S_t \cap M_-$ in $H_2$ for $t>\hat t$ and $t$ arbitrarily close to $\hat t$. 
Note that a single band sum must occur as $t$ crosses the value $\hat t$. 
This pushes a band of $S_t$ across $J$. 
The side, say $M_+$, on which the band leaves must be the side containing the compressing disk in $H_1$ and the other side $M_-$ which receives the band must be where the first compressing disk for $H_2$ appears. 
Then, at such a level $t$ with $t<\hat t$ very close to $\hat t$, there cannot be any compressing disks in $S_t \cap M_-$, since $S$ is strongly irreducible, and similarly for $t>\hat t$ and $t$ very close to $\hat t$, there are no compressing disks for $S_t \cap M_+$. 
This gives the second case of the theorem. 

Finally the third case occurs when for $t>\hat t$ and $t$ arbitrarily close to $\hat t$, there are no compressing disks for either $S_t \cap M_+$ or $S_t \cap M_-$.
\end{proof}

If we consider the \textit{Hempel distance} of a Heegaard surface, which is a natural generalization of the strong irreducibility condition and was introduced in \cite{He1}, we obtain the following corollary. 

\begin{cor}
Suppose that $J$ is separating and incompressible and $S$ is a Heegaard splitting which has Hempel distance at least $4$. Then only the third possibility can occur. 
\end{cor}

\begin{proof}
The first possibility in the theorem contradicts Hempel distance at least $3$. 
Recall that the second case occurs when a single band sum of $S_t$ across $J$ at the critical level $\hat t$ produces a compressing disk $D_1$ in $H_1$ for $S_t$ and $t<\hat t$, whereas there is a compressing disk $D_2$ for $S_{t^\prime}$ in $H_2$ for $t^\prime >\hat t$. 
There are either one or two curves of $S_t \cap J$ involved with the band sum. 
After the band sum, we get a new family of curves which can be pushed off the old family. 
But then we see that there is a compressing disk $D_1$ for $S_t$ in $H_1$ disjoint from $S_t \cap J$ for $t<\hat t$ and similarly a compressing disk $D_2$ for $S_{t^\prime}$ in $H_2$ for $t^\prime >\hat t$. 
We conclude that $\partial D_1$ is disjoint from $S_t \cap J$ which can be made disjoint from $S_{t^\prime} \cap J$ which is disjoint from $\partial D_2$. 
This contradicts the Hempel distance of $S$ being at least $4$. 
\end{proof}

\begin{rem}

Notice that although both sides of $S_t$ are incompressible in the conclusion of the corollary, they might not be boundary incompressible. However both these subsurfaces can be boundary compressed to form spanning surfaces for the regions on either side of $J$. 

\end{rem}

\begin{rem}
In the situation of the corollary, we can estimate the genus of $S$ by adding minimal genera of incompressible and boundary incompressible 
surfaces on each side of $J$, of course taking into account the number of boundary curves. 
So suppose that $J$ is a separating incompressible surface. 
Amongst all incompressible and boundary incompressible surfaces $A_+, A_-$  in $M_+, M_-$, choose the ones which minimize $h=|\chi(A_+)| + |\chi(A_-)| + 2(2n-m-1)$ where $n \ge m$ and say $n = |\partial A_+|, m=|\partial A_-|$ are the numbers of boundary curves. 
Then $h$ gives a lower bound for the absolute value of the Euler characteristic of $S$ and hence gives a convenient bound for the Heegaard genus. 
\end{rem}

\end{document}